\documentclass[12pt]{amsart}
\usepackage[utf8]{inputenc}
\usepackage{amssymb}
\usepackage{graphicx}
\usepackage{graphicx,color}
\usepackage{latexsym}
\usepackage[all]{xy}
\usepackage{mathrsfs}
\usepackage{enumerate}
\usepackage{amssymb}
\usepackage{tikz}
\usepackage{breqn}

\makeatletter
\renewcommand{\subsection}{\@startsection
{subsection}{2}{0mm}{\baselineskip}{-0.25cm}
{\normalfont\normalsize\em}}
\makeatother

\def\negei{\mathbf e_i}

\def\negP{\mathbf P}
\def\negQ{\mathbf Q}
\def\negalpha{\text{\boldmath$\alpha$}}
\def\neglambda{\text{\boldmath$\lambda$}}

\def\neg1{\text{\boldmath$1$}}
\def\negbeta{\text{\boldmath$\beta$}}
\def\neggamma{\text{\boldmath$\gamma$}}
\def\negeta{\text{\boldmath$\eta$}}

\def\neggamma{\text{\boldmath$\gamma$}}

\def\negeta{\text{\boldmath$\eta$}}

\def\neg1{\text{\boldmath$1$}}

\def\hH{\widehat{H}}
\def\hS{\widehat{S}}
\def\hGamma{\widehat{\Gamma}}
\def\cC{\mathcal C}

\def\cL{\mathcal L}

\def\cX{\mathcal X}

\def\NN{\mathbb{N}}

\def\ZZ{\mathbb{Z}}

\def\FF{\mathbb{F}}

\def\Fq{{\mathbb{F}_q}}

\DeclareMathOperator{\divv}{div}
\DeclareMathOperator{\divvp}{div_\infty}

\DeclareMathOperator{\lub}{lub}

\newtheorem{theorem}{Theorem}[section]
\newtheorem{proposition}[theorem]{Proposition}
\newtheorem{corollary}[theorem]{Corollary}
\newtheorem{lemma}[theorem]{Lemma}

{\theoremstyle{definition}
\newtheorem{definition}[theorem]{Definition}
\newtheorem{example}[theorem]{Example}

}

{\theoremstyle{remark}

}

\title[Generalized Weierstrass semigroups and Riemann-Roch spaces]{Generalized Weierstrass semigroups and Riemann-Roch spaces for certain curves with separated variables}

\author[W. Ten\'orio \and G. Tizziotti]{Wanderson Ten\'orio \and Guilherme Tizziotti}
 \address{Universidade Federal de Uberl\^andia (UFU), Faculdade de Matem\' atica, Av.~J.~N. \' Avila~2121, 38408-902, Uberl\^andia, MG, Brazil}
  \email{dersonwt@yahoo.com.br, guilhermect@ufu.br}

\begin{document}


\maketitle

\begin{abstract}
In this work we study the generalized Weierstrass semigroup $\widehat{H} (\mathbf{P}_m)$ at an $m$-tuple $\mathbf{P}_m = (P_{1}, \ldots , P_{m})$ of rational points on certain curves admitting a plane model of the form $f(y) = g(x)$ over $\mathbb{F}_{q}$, where ${f(T),g(T)\in \mathbb{F}_q[T]}$. In particular, we compute the generating set $\widehat{\Gamma}(\mathbf{P}_m)$ of $\widehat{H} (\mathbf{P}_m)$ and, as a consequence, we explicit a basis for Riemann-Roch spaces of divisors with support in $\{P_{1}, \ldots , P_{m}\}$ on these curves, generalizing results of Maharaj, Matthews, and Pirsic in \cite{maharaj}.
\end{abstract}

\section{Introduction}

Let $\mathcal{X}$ be a nonsingular, projective, geometrically irreducible algebraic curve of positive genus defined over a finite field $\mathbb{F}_q$ with $q$ elements and let $\mathbb{F}_q(\mathcal{X})$ be its associated function field. For $h \in \mathbb{F}_q(\mathcal{X})^\times$,  $\divv(h)$ and $\divvp(h)$ will stand respectively for the divisor and the pole divisor of $h$. We denote $\mathbb{N}_{0} = \mathbb{N} \cup \{0\}$, where $\mathbb{N}$ is the set of positive integers. Let $\mathbf{Q} = (Q_{1}, \ldots , Q_{m})$ be an $m$-tuple of pairwise distinct rational points on $\mathcal{X}$. The set
$$
H(\mathbf{Q}) := \left\{(a_{1}, \ldots, a_{m}) \in \mathbb{N}_{0}^ {m} \mbox{ : } \exists \ h \in \mathbb{F}_q(\mathcal{X})^\times \mbox{ with } \divvp(h) = \sum_{i=1}^ {m} a_{i}Q_{i} \right\}
$$
is a subsemigroup of $\NN_0^m$ with respect to addition, called the \textit{(classical) Weierstrass semigroup} of $\cX$ at the $m$-tuple $\mathbf{Q}$.

The case $m=1$---known as the Weierstrass semigroup of $\cX$ at $Q_1$---is a classical object in algebraic geometry which is related to many theoretical and applied topics; see e.g. \cite{centina}, \cite{GarciaKimLax}, \cite{vanlint}, \cite{homma} and \cite{stichtenoth2}. For $m\geq 2$, the firsts progresses in the understanding of these structures are due to Kim \cite{kim} and Homma \cite{homma2}, who studied several properties of Weierstrass semigroups at a pair of points. It allowed computing explicitly such semigroups at pairs for several curves as Hermitian, Suzuki, Norm-trace, and curves defined by Kummer extensions, which are of interest in coding theory; we refer the reader to \cite{gretchen}, \cite{gretchenSuz}, \cite{tizziotti} and \cite{masuda}. The general case was considered in several works as \cite{ballicokim}, \cite{carvalho2} and \cite{gretchen1}; see also \cite{carvalho} for a survey exposition on this topic. In particular, Matthews \cite{gretchen1} introduced a notion of generating set of $H(\mathbf{Q})$ that allows to constructing the semigroup from a finite number of elements. Matthews's approach was used to compute the generating set of classical Weierstrass semigroups at some collinear points on the Hermitian and Norm-trace curves in \cite{gretchen1} and \cite{gretchen2}, respectively. A general study of the generating set of classical Weierstrass semigroups at several points on certain curves of the form $f(y)=g(x)$, which include many classic curves over finite fields, was made in \cite{CT} by Castellanos and Tizziotti.

In \cite{delgado}, Delgado proposed a different generalization of Weierstrass semigroups at a single point to several points on a curve. His original approach was on curves over algebraically closed fields, but later Beelen and Tutas \cite{BR2} studied such objects for curves defined over finite fields, where they presented many properties of these semigroups. This generalization is as follows: given an $m$-tuple $\mathbf{Q}=(Q_1,\ldots,Q_m)$ of pairwise distinct rational points on $\cX$, we let $R_{\mathbf{Q}}$ denote the ring of functions of $\mathcal{X}$ having poles only on the set $\{Q_1,\ldots,Q_m\}$. The \emph{generalized Weierstrass semigroup} of $\cX$ at $\mathbf{Q}$ is the set
\begin{equation}\label{generalized WS}
 \widehat{H} (\mathbf{Q}):=\{(-v_{Q_1}(h),\ldots,-v_{Q_m}(h))\in \mathbb{Z}^m \ : \ h\in R_{\mathbf{Q}}\backslash\{0\} \},
 \end{equation}
where $v_{Q_i}$ stands for the valuation of the function field $\mathbb{F}_q(\mathcal{X})$ associated with $Q_i$. The classical and generalized Weierstrass semigroups are related by $ H(\mathbf{Q}) = \widehat{H}(\mathbf{Q}) \cap \NN_0^m$ provided that $q \geq m$;  see \cite[Proposition 6]{BR2}. 

Moyano-Fern\' andez, Ten\' orio, and Torres \cite{TT} studied the generalized Weierstrass semigroups at several points on a curve over a finite field. Motivated by the description of the classical Weierstrass semigroups at several points given by Matthews \cite{gretchen1}, they characterized the semigroups $\widehat{H} (\mathbf{Q})$ in terms of the \emph{absolute maximal elements} of $\widehat{H} (\mathbf{Q})$   (see Definition 2.2) introduced by Delgado \cite{delgado}, providing a generating set for $\widehat{H} (\mathbf{Q})$ in the sense of \cite{gretchen1}. Despite this generating set---formed by absolute maximal elements of $\widehat{H} (\mathbf{Q})$---is an infinite set, they showed it can be computed through a finite number of elements in $\widehat{H}(\mathbf{Q})$. They also proved that the absolute maximal elements of $\widehat{H}(\mathbf{Q})$ carry information on the corresponding Riemann-Roch spaces of every divisor with support contained in $\{Q_1,\ldots,Q_m\}$, and consequently that the arithmetic of these divisors can be deduced from the generating set of $\widehat{H} (\mathbf{Q})$.

The main focus of this paper is to study the generalized Weierstrass semigroup $\widehat{H} (\mathbf{P}_m)$ at an $m$-tuple $\mathbf{P}_m=(P_1,\ldots,P_m)$ of some rational points on certain curves with a plane model of the form $f(y) = g(x)$ over $\mathbb{F}_{q}$, denoted by $\cX_{f,g}$, where $f(T),g(T)\in \mathbb{F}_q[T]$. We establish a connection between the absolute maximal elements of generalized Weierstrass semigroups and the concept of \emph{discrepancy} (see Definition \ref{defi discrepancy}) presented by Duursma and Park in \cite{duursma}, which allows us to determine the generating set of $\widehat{H} (\mathbf{P}_m)$ for this type of curves, as in \cite{CT} for the classical approach. As a consequence, we describe a basis of Riemann-Roch spaces of certain divisors on this class of curves. We also introduce the notion of the supported floor of a divisor and explore its relationship with the generating set of generalized Weierstrass semigroups. 

This paper is organized as follows. Section 2 contains basic concepts and results on  generalized Weierstrass semigroups and discrepancies. In Section 3 we introduce the concept of the supported floor of a divisor and study the relations of this notion with the generalized Weierstrass semigroups. In Section 4 we study the generalized Weierstrass semigroup $\widehat{H} (\mathbf{P}_m)$ at several points on the curves $\mathcal{X}_{f,g}$. Finally, Section 5  is devoted to providing expressions for the dimension and basis of Riemann-Roch spaces of certain divisors on $\mathcal{X}_{f,g}$, as well as for the floor and the supported floor of such divisors.

\section{Background and Preliminary Results}

 Let $\mathcal{X}$ be a (nonsingular, projective, geometrically irreducible) curve of positive genus defined over $\mathbb{F}_q$. Let $\mbox{Div}(\mathcal{X})$ be the set of divisors on $\mathcal{X}$. For $G\in\mbox{Div}(\mathcal{X})$, let $\mathcal{L}(G):= \{ h \in \mathbb{F}_{q}(\mathcal{X})^\times \mbox{ : } \divv(h) + G \geq 0 \} \cup \{ 0 \}$ be the Riemann-Roch space associated to $G$. The dimension of $\mathcal{L}(G)$ as an $\mathbb{F}_{q}$-vector space will be denoted by $\ell(G)$.

Throughout this section $\mathbf{Q}$ will stand for the $m$-tuple $(Q_1, \ldots, Q_m)$ of  $m\geq 2$ distinct pairwise rational points on $\mathcal{X}$. For $\neggamma=(\gamma_1,\ldots,\gamma_m)\in \mathbb{Z}^m$,  $D_\neggamma$ will denote the divisor $ \gamma_1Q_1+\cdots+\gamma_mQ_m$ on $\cX$.

\subsection{Generalized Weierstrass semigroups} In what follows we briefly review some important facts concerning the generalized Weierstrass semigroup $ \widehat{H} (\mathbf{Q})$ defined in (\ref{generalized WS}). For ${\negalpha=(\alpha_1,\ldots,\alpha_m)\in \mathbb{Z}^m}$ and $i\in \{1,\ldots,m\}$, we let
$$\nabla_i^m(\negalpha):=\{\negbeta=(\beta_1,\ldots,\beta_m)\in \hH(\negQ) \ : \ \beta_i=\alpha_i \ \text{and }\beta_j\leq \alpha_j \ \text{for }j\neq i\}.$$
Such sets play an important role in the characterization of $\hH(\mathbf{Q})$ as follows.

\begin{proposition}[\cite{TT}, Proposition 2.1] \label{prop 2.1.3 da tese}
Let $\negalpha \in \mathbb{Z}^m$ and assume that $q \geq m$. Then,
\begin{enumerate}[\rm (1)]
\item $\negalpha \in \widehat{H} (\mathbf{Q})$ if and only if $\ell(D_{\negalpha}) = \ell(D_{\negalpha} - Q_i)+1$, for all $i \in \{1,\ldots,m\}$;
\item $\nabla_{i}^{m}(\negalpha) = \emptyset$ if and only if $\ell(D_{\negalpha})=\ell(D_{\negalpha} - Q_i)$.
\end{enumerate}
\end{proposition}

Let $\negalpha=(\alpha_1,\ldots,\alpha_m)\in \mathbb{Z}^m$. For a nonempty subset $J\subsetneq \{1,\ldots,m\}$, define
$$\nabla_J(\negalpha):=\{\negbeta=(\beta_1,\ldots,\beta_m)\in \hH(\negQ) \ : \ \beta_j=\alpha_j \ \text{for }j\in J, \ \text{and }\beta_i<\alpha_i  \ \text{for }i\not\in J\}.$$

\begin{definition} \label{abs maximal} An element $\negalpha\in \widehat{H} (\mathbf{Q})$ is called \textit{absolute maximal} if ${\nabla_J(\negalpha)=\emptyset}$ for every nonempty $J\subsetneq \{1,\ldots,m\}$. The set of absolute maximal elements in $\widehat{H} (\mathbf{Q})$ will be denoted by $\widehat{\Gamma}(\mathbf{Q})$.
\end{definition}

\begin{proposition} [\cite{TT}, Proposition 3.2] \label{absmax} Let $\negalpha\in \hH(\mathbf{Q})$ and assume that $q \geq m$. Then, the following statements are equivalent:
\begin{enumerate}[\rm (i)]
\item $\negalpha\in \hGamma(\mathbf{Q})$;
\item $\nabla_i^m(\negalpha)=\{\negalpha\}$ for all $i\in \{1,\ldots,m\}$;
\item  $\nabla_i^m(\negalpha)=\{\negalpha\}$ for some $i\in \{1,\ldots, m\}$;
\item $\ell(D_{\negalpha})=\ell(D_{\negalpha} - \sum_{i=1}^m Q_i)+1$.
\end{enumerate}
\end{proposition}

Given a finite subset $\mathcal{B}\subseteq \ZZ^m$,  we define the \emph{least upper bound} ($\lub$) of $\mathcal{B}$ by
$$\mbox{lub}(\mathcal{B}):=(\max\{\beta_1\mbox{ : } \negbeta \in \mathcal{B} \},\ldots,\max\{\beta_m\mbox{ : } \negbeta \in \mathcal{B} \})\in \ZZ^m.$$

\begin{theorem}[\cite{TT}, Theorem 3.4] Assume that $q \geq m$. The generalized Weierstrass semigroup of $\cX$ at $\mathbf{Q}$ can be written as
$$\hH(\mathbf{Q})=\{\lub(\{\negbeta^1,\ldots,\negbeta^m\}) \ : \ \negbeta^1,\ldots,\negbeta^m\in \hGamma(\mathbf{Q})\}.$$
\end{theorem}

Note that the previous theorem shows that the set $\hGamma(\mathbf{Q})$ determines $\hH(\mathbf{Q})$ in terms of least upper bounds. In this way, $\hGamma(\mathbf{Q})$ can be seen as a \textit{generating set} of $\hH(\mathbf{Q})$ in the sense of \cite{gretchen1}. We observe that, unlike the case of generating set for classical Weierstrass semigroups of \cite{gretchen1}, the set $\hGamma(\mathbf{Q})$ is not finite. Nevertheless, it is finitely determined as follows.\\

For $i=2,\ldots,m$, let $a_i$ be the smallest positive integer $t$ such that ${tQ_{i}-tQ_{i-1}}$ is a principal divisor on $\cX$. We can thus define the region
$$\cC_m:=\{\negalpha=(\alpha_1, \ldots , \alpha_m)\in \ZZ^m \ : \ 0\leq \alpha_i< a_i \ \mbox{for } i=2,\ldots,m\}.$$ \label{teste}
Let $\negeta^i = (\eta_1^i, \ldots, \eta_m^i)\in \ZZ^m$ be the $m$-tuple whose $j$-th coordinate is
$$
\eta^i_j=\left \lbrace \begin{array}{rcl}
-a_i & , & \mbox{if } j=i-1\\
a_i & , & \mbox{if } j=i\\
0 & , & \mbox{otherwise.}\\
\end{array}\right.
$$
Defining ${\Theta_m:=\{ b_1 \negeta^1+ \ldots + b_{m-1} \negeta^{m-1}\in \ZZ^m \ : b_i \in \ZZ \ \mbox{for } i=1,\ldots,m-1\}}$, we can state the following concerning the absolute maximal elements in generalized Weierstrass semigroups at several points.

\begin{theorem}[\cite{TT}, Theorem 3.7]\label{maximals} Assume that $q \geq m$. The following holds
$$\hGamma(\mathbf{Q})=(\hGamma(\mathbf{Q})\cap \cC_m)+\Theta_m.$$
\end{theorem}

Notice that, since $\hGamma(\mathbf{Q})\cap \cC_m$ is finite and $\Theta_m$ is finitely generated, $\hGamma(\mathbf{Q})$ is determined by a finite number of elements in $\hH(\negQ)$. We also observe that the set $\cC_m$ and the elements $\negeta^i$'s are slightly modified from that defined in \cite[Section 3]{TT}. This modification does not affect the conclusions and its purpose is to make our subsequent computations simpler.\\

The absolute maximal elements in $\hH(\negQ)$ are also related to Riemann-Roch spaces as below. For $\negalpha=(\alpha_1, \ldots , \alpha_m)\in \ZZ^m$, define
\begin{equation}\label{Gamma alpha}
\hGamma_{\mathbf{Q}}(\negalpha):=\{\negbeta\in \hGamma(\mathbf{Q}) \ : \ \negbeta\leq \negalpha\},
\end{equation} 
where $\negbeta\leq \negalpha$ means that $\beta_j\leq \alpha_j$ for $j=1,\ldots,m$. Note that $\hGamma_{\mathbf{Q}}(\negalpha)$ is a finite subset of $\ZZ^m$. For $i\in \{1,\ldots,m\}$ and $\negbeta,\negbeta'\in \hGamma_{\mathbf{Q}}(\negalpha)$, define
$$\negbeta\equiv_i\negbeta' \ \mbox{if and only if} \ \beta_i=\beta_i'.$$
A straightforward verification shows that $\equiv_i$ is an equivalence relation, and so we may consider the quotient set $\hGamma_{\mathbf{Q}}(\negalpha)/\equiv_i$ as follows.

\begin{theorem} [\cite{TT},~Theorem 3.5] \label{dimension} Let $\negalpha\in \ZZ^m$ and $i\in\{1,\ldots m\}$, and assume that $q\geq m$. Then,
$$\ell(D_{\negalpha})=\#\left(\hGamma_{\mathbf{Q}}(\negalpha)/\equiv_i\right).$$
\end{theorem}

\begin{corollary} [\cite{TT},~Corollary 3.6] \label{basis} Under the same assumptions as in the previous result, let $(\hGamma_{\mathbf{Q}}(\negalpha)/\equiv_i) \ =\{[\negbeta^1]_i,\ldots,[\negbeta^{\ell(D_{\negalpha})}]_i\}$. For each $1\leq j\leq\ell(D_\negalpha)$, choose one $f_j\in \FF(\cX)$ such that 
$$(-v_{Q_1}(f_j),\ldots,-v_{Q_m}(f_j))=\negbeta^j.$$
Then  $\{f_1,\ldots,f_{\ell(D_{\negalpha})}\}$ constitutes a basis for the Riemann-Roch space $\cL(D_{\negalpha})$.
\end{corollary}

\subsection{Discrepancy} In \cite[Section 5]{duursma}, Duursma and Park introduced the concept of discrepancy. In the remaining of this section, we establish a connection between discrepancies and absolute maximal elements in generalized Weierstrass semigroups at several points.

\begin{definition}\label{defi discrepancy}
Let $P$ and $Q$ be distinct rational points $P$ and $Q$ on $\mathcal{X}$. A divisor $A \in \mbox{Div}(\mathcal{X})$ is called a \textit{discrepancy} with respect to $P$ and $Q$ if $\mathcal{L}(A)\neq \mathcal{L}(A-Q)$ and $\mathcal{L}(A-P)=\mathcal{L}(A-P-Q)$.
\end{definition}

The following result states an equivalent formulation for the absolute maximal property regarding discrepancies. It will be especially useful to compute absolute maximal elements in Section 4.

\begin{proposition} \label{prop equiv discrepancia}
Let $\negalpha= (\alpha_1, \ldots , \alpha_{m}) \in \ZZ^m$ and assume that $q\geq m$. The following statements are equivalent:
\begin{enumerate}[\rm (i)]
\item $\negalpha \in \hGamma(\mathbf{Q})$;
\item $D_{\negalpha}$ is a discrepancy with respect to any pair of distinct points in $\{Q_1,\ldots,Q_{m}\}$.
\end{enumerate}
\end{proposition}
\begin{proof}
$(i) \Rightarrow (ii):$ Note that, by Proposition \ref{prop 2.1.3 da tese}(1), $\ell(D_{\negalpha}) = \ell(D_{\negalpha} - Q_i)+1$ for all $i \in \{1,\ldots,m\}$, since $\negalpha \in \widehat{H} (Q_1, \ldots , Q_{m})$. So $\mathcal{L}(D_{\negalpha})\neq \mathcal{L}(D_{\negalpha}-Q_i)$ for all $i \in \{1,\ldots,m\}$. Now, suppose that $\mathcal{L}(D_{\negalpha}-Q_i) \neq \mathcal{L}(D_{\negalpha}-Q_i - Q_j)$ for some $j \neq i$. Hence, by Proposition \ref{prop 2.1.3 da tese}(2), it follows that $\nabla_{j}^{m}(\negalpha-\negei) \neq \emptyset$, where $\negei = (0,\ldots,0,1,0,\ldots,0) \in \ZZ^m$. Let $\negbeta=(\beta_1,\ldots,\beta_m) \in \nabla_{j}^{m}(\negalpha-\negei)$. As $i \neq j$, we have $\negbeta\in \nabla_{j}^{m}(\negalpha)$. Consider ${J = \{t \in \{1,\ldots,m\} \ : \ \beta_t = \alpha_t\}}$. Note that $J \neq \emptyset$ since $j \in J$. Furthermore, we have $J \subsetneq \{1,\ldots,m\}$ since $i \notin J$. We thus conclude that $\negbeta \in \nabla_{J}(\negalpha)$, contrary to $\negalpha\in\hGamma(\negQ)$. Consequently, $\mathcal{L}(D_{\negalpha}-Q_i) = \mathcal{L}(D_{\negalpha}-Q_i - Q_j)$ for all distinct $i , j \in \{1,\ldots,m\}$ and therefore $D_{\negalpha}$ is a discrepancy with respect to any distinct pair of points ${Q_i,Q_j\in \{Q_1,\ldots,Q_{m}\}}$.

$(ii) \Rightarrow (i):$ Observe that $\negalpha\in \hH(\negQ)$. Suppose now that $\nabla_{J}(\negalpha) \neq \emptyset$ for some $J \subsetneq \{1,\ldots,m\}$ with $J \neq \emptyset$ and let $\negbeta \in \nabla_{J}(\negalpha)$. Thus, for $i \in \{1,\ldots,m\} \setminus J$ and $j \in J$, we have $\negbeta \in \nabla_{j}^{m}(\negalpha-\mathbf{e}_i)$. Proposition \ref{prop 2.1.3 da tese}(2) yields $\ell(D_{\negalpha} - Q_i) \neq \ell(D_{\negalpha} - Q_i - Q_j )$, which contradicts the fact that $D_{\negalpha}$ is a discrepancy with respect to $Q_i$ and $Q_j$ for any pair of distinct rational points $Q_i,Q_j\in \{Q_1,\ldots,Q_{m}\}$. Therefore, $\negalpha$ is an absolute maximal element of $\hH(\negQ)$.
\end{proof}

\section{The supported floor of divisors $D_\negalpha$}

In this section we introduce the notion of the supported floor of a divisor. Roughly speaking, the supported floor of a divisor $D$ with $\ell(D)>0$ will be the divisor of minimum degree with support contained in the support of $D$ and whose the associated Riemann-Roch spaces coincide. We will explore the outcomes of this idea and study its relations with the generalized Weierstrass semigroups. As in the previous sections, $\cX$ denotes a curve over $\Fq$ and $\negQ$ an $m$-tuple of distinct rational points $Q_1,\ldots,Q_m$ on $\cX$. Recall that $D_\negalpha$ for $\negalpha\in \ZZ^m$ represents the divisor $\alpha_1Q_1+\cdots+\alpha_mQ_m$ on $\cX$.\\

The concept of \textit{floor} of a divisor $G$ on $\cX$ with $\ell(G)>0$ was introduced in \cite{maharaj}. Denoted by $\lfloor G\rfloor$, the floor of $G$ is the unique divisor $G'$ on $\cX$ of minimum degree such that $\cL(G)=\cL(G')$. Recall that the greatest common divisor of $D=\sum_{P\in \cX} v_P(D) P$ and $D'=\sum_{P\in \cX} v_P(D') P$ is $\mbox{gcd}(D,D'):=\sum_{P\in \cX} \min\{v_P(D),v_P(D')\} P$.

\begin{theorem} [\cite{maharaj}, Theorem 2.6] \label{floor} Let $G$ be a divisor on $\cX$ with $\ell(G)>0$ and let $h_1,\ldots,h_s\in \cL(G)$ be a spanning set for $\cL(G)$. Then,
$$\lfloor G \rfloor=-\gcd(\divv(h_i) \ : \ i=1,\ldots,s).$$
\end{theorem}

Let $\negalpha\in\ZZ^m$ and suppose that $\ell(D_{\negalpha})>0$. If $D_{\negalpha}$ is an effective divisor, by \cite[Theorem~2.5]{maharaj}, the support of $\lfloor D_{\negalpha} \rfloor$ is a subset of $\{Q_1,\ldots,Q_m\}$. However, if $D_{\negalpha}$ is not effective, then the support of $\lfloor D_{\negalpha} \rfloor$ may not be necessarily contained in $\{Q_1,\ldots,Q_m\}$, and that is the reason we define the supported floor.

Given $\negalpha\in\ZZ^m$, let us consider the set
$$
T_{\negalpha}:= \{ \negbeta\in \ZZ^m \mbox{ : }  \negbeta \leq  \negalpha \mbox{ and } \ell( D_{\negbeta})=\ell( D_{\negalpha}) \}.
$$
Notice that $T_{\negalpha}$ is nonempty since $ \negalpha \in T_{\negalpha}$. In addition, as $\ell(D_{\negalpha})>0$, we have that $T_{\alpha}$ is a finite set. 
For $\negbeta=(\beta_1, \ldots, \beta_m)\in\ZZ^m$, the \textit{norm} of $\negbeta$ is the integer $|\negbeta|:=\beta_1 + \cdots + \beta_m$.

\begin{proposition} \label{prop m floor}
There is a unique $\negbeta \in T_{\negalpha}$ with minimum norm in ${\{|\neggamma| \ : \ \neggamma\in T_\negalpha\}}$.
\end{proposition}
\begin{proof}
Since $T_{\negalpha}$ is finite and the norm is an integer, $T_{\negalpha}$ has an element whose the norm is minimum in $\{|\neggamma| \ : \ \neggamma\in T_\negalpha\}$. In order to prove that such an element is unique, suppose that there exist $\negbeta,\negbeta'\in T_{\negalpha}$, with ${\negbeta \neq \negbeta'}$, having the minimum norm in $\{|\neggamma| \ : \ \neggamma\in T_\negalpha\}$. As $\negbeta,\negbeta'\in T_\negalpha$, it follows that $\cL(D_{\negbeta}) = \cL(D_{\negalpha}) = \cL(D_{\negbeta'})$. In addition, we have $\cL(D_{\negalpha})=\cL(D_{\negbeta})\cap \cL(D_{\negbeta'})=\cL(\gcd(D_{\negbeta},D_{\negbeta'}))$. So, writing $\gcd(D_{\negbeta},D_{\negbeta'})=D_\neggamma$ for some ${\neggamma = (\gamma_1, \ldots , \gamma_m)\in \ZZ^m}$, we obtain $|\neggamma|<|\negbeta|=|\negbeta'|$, which gives us a contradiction since $\neggamma \in T_{\negalpha}$.
\end{proof}

\begin{definition} \label{P piso}
Let $\negalpha\in \ZZ^m$ such that $\ell(D_{\negalpha})>0$. The \emph{supported floor} of $D_{\negalpha}$ at $\negQ$ (or simply $\negQ$-\emph{floor} of $D_\negalpha$ for short) is the divisor $D_{\negbeta}$, where $\negbeta \in T_{\negalpha}$ is the element of minimum norm in $\{|\neggamma| \ : \ \neggamma\in T_\negalpha\}$ as in Proposition \ref{prop m floor}. The $\mathbf{Q}$-floor of $D_{\negalpha}$ will be denoted by $\lfloor D_{\negalpha} \rfloor_{\negQ}$.
\end{definition}

\begin{lemma}\label{lema P piso}
Let $\negalpha\in \mathbb{Z}^m$ and assume that $q\geq m$. 
\begin{enumerate}[\rm (1)]
\item If $\lfloor D_\negalpha\rfloor_\negQ=D_\negbeta$, then $\negbeta\in \hH(\negQ)$. In particular, $\lfloor D_{\negalpha} \rfloor_{\mathbf{Q}} = D_{\negalpha}$, provided $\negalpha\in \hH(\mathbf{Q})$; 

\item If $\negbeta, \negbeta' \in T_{\negalpha} \cap \hH(\mathbf{Q})$, then $\negbeta = \negbeta'$;

\item If $\negbeta \in T_{\negalpha}$, then $\lfloor D_{\negbeta} \rfloor_{\mathbf{Q}} = \lfloor D_{\negalpha} \rfloor_{\mathbf{Q}}$.
\end{enumerate}
\end{lemma}

\begin{proof}
$(1).$ It follows directly from Proposition \ref{prop 2.1.3 da tese}(1).

$(2).$ Suppose, on the contrary, that $\negbeta \neq \negbeta'$. As $\negbeta, \negbeta' \in T_{\negalpha}$, we have $\mathcal{L}(D_{\negalpha})= \mathcal{L}(D_{\negbeta})=\mathcal{L}(D_{\negbeta'})$. Let $\neglambda = (\lambda_1, \ldots , \lambda_m)$, where $\lambda_i = \min \{ \beta_i , \beta_i ' \} $ for $i=1,\ldots,m$. Thus, $\neglambda \neq \negbeta$ or $\neglambda  \neq \negbeta'$ since $\negbeta\neq \negbeta'$. Furthermore, since $\cL(D_{\neglambda}) = \mathcal{L}(D_{\negbeta}) \cap \mathcal{L}(D_{\negbeta'})$, we have $\cL(D_{\neglambda}) = \mathcal{L}(D_{\negbeta}) = \mathcal{L}(D_{\negbeta'})$. 
Proposition \ref{prop 2.1.3 da tese}(1) now yields a contradiction because $\neglambda\leq \negbeta$, $\neglambda\leq \negbeta'$, and $\negbeta,\negbeta'\in \hH(\negQ)$.

$(3).$  Note that $\negbeta\in T_\negalpha$ implies $T_\negbeta\subseteq T_\negalpha$. The assertion thus follows by the items above.
\end{proof}

The next result exploits the relation between $\mathbf{Q}$-floor and the elements of $\hGamma(\negQ)$.

\begin{theorem}\label{lubfloor}
Let $\negalpha\in \ZZ^m$ with $\ell(D_{\negalpha})>0$ and assume that $q\geq m$. Then,
$$\lfloor D_{\negalpha} \rfloor_{\mathbf{Q}} = D_{\lub(\hGamma_{\negQ}(\negalpha))}. $$
\end{theorem}
\begin{proof}  Let $\negalpha\in \ZZ^m$ with $\ell(D_{\negalpha})>0$ and let $\neggamma = \text{lub}(\hGamma_{\mathbf{Q}}(\negalpha))\in \ZZ^m$. Hence $\neggamma \in \hH(\mathbf{Q})$ and $\neggamma\leq \negalpha$. Theorem \ref{dimension} now gives $\ell(D_{\neggamma})=\ell(D_{\negalpha})$, since ${\hGamma_{\mathbf{Q}}(\neggamma)=\hGamma_\negQ(\negalpha)}$. Consequently, $\neggamma \in T_{\negalpha}$ and the result follows from the items (1) and (3) in the previous lemma.
\end{proof}

Given $f\in R_{\mathbf{Q}}\backslash\{0\}$, let us denote $\rho(f):=(-v_{Q_1}(f),\ldots,-v_{Q_m}(f))$. For each $\negbeta\in \hGamma_{\mathbf{Q}}(\negalpha)$, choose one function $f_\negbeta\in R_{\mathbf{Q}}$ such that $\rho(f_\negbeta)=\negbeta$. By abuse of notation, we will write $\rho^{-1}(\hGamma_{\mathbf{Q}}(\negalpha))$ for $\{f_\negbeta \ : \ \negbeta\in \hGamma_{\mathbf{Q}}(\negalpha)\}$. The next result provides a sufficient condition for the equality between $\lfloor D_{\negalpha} \rfloor$ and $\lfloor D_{\negalpha} \rfloor_{\mathbf{Q}}$.

\begin{proposition} \label{mflooreqfloor} Let $\negalpha\in \ZZ^m$. Then, $\lfloor D_{\negalpha} \rfloor = \lfloor D_{\negalpha} \rfloor_{\mathbf{Q}}$ provided that $f_\negbeta\in \rho^{-1}(\hGamma_{\mathbf{Q}}(\negalpha))$ has no zero in $\cX\backslash \{Q_1,\ldots,Q_m\}$. 
\end{proposition}
\begin{proof}
The basic idea is to observe that $\rho^{-1}(\hGamma_{\mathbf{Q}}(\negalpha))$ is a spanning set of $\cL(D_{\negalpha})$; indeed, by Corollary \ref{basis}, $\rho^{-1}(\hGamma_{\mathbf{Q}}(\negalpha))$ contains a basis of $\cL(D_{\negalpha})$ and so spans $\cL(D_{\negalpha})$. Now, by hypothesis, we have $v_P(f_\negbeta)= 0$ for $P\in \cX\backslash \{Q_1,\ldots,Q_m\}$ and for $\negbeta\in \hGamma_{\mathbf{Q}}(\negalpha)$. 
Hence Theorem \ref{floor} yields
$$\begin{array}{rcl}
 \lfloor D_{\negalpha} \rfloor & = & -\gcd(\divv(f_\negbeta) \ : \ \negbeta\in \hGamma_{\mathbf{Q}}(\negalpha))\\
 & = & -\gcd\left(\sum_{P\in \cX} v_P(f_\negbeta)P \ : \ \negbeta\in \hGamma_{\mathbf{Q}}(\negalpha)\right) \\
 & = & -\sum_{P\in \cX} \min\{v_P(f_\negbeta) \ : \ \negbeta\in \hGamma_{\mathbf{Q}}(\negalpha)\} P \\
 & = & \sum_{P\in \cX} \max\{-v_P(f_\negbeta) \ : \ \negbeta\in \hGamma_{\mathbf{Q}}(\negalpha)\} P \\
 & = & \sum_{i=1}^m \max\{-v_{Q_i}(f_\negbeta) \ : \ \negbeta\in \hGamma_{\mathbf{Q}}(\negalpha)\} Q_i,
\end{array}$$
which gives us $\lfloor D_{\negalpha} \rfloor_{\mathbf{Q}} = \lfloor D_{\negalpha} \rfloor$ by Theorem \ref{lubfloor}.
\end{proof}

\section{ Generalized Weierstrass semigroups of certain curves of the form $f(y) = g(x)$} \label{section curves}

Let $\mathcal{X}_{f,g}$ be a curve over $\mathbb{F}_{q}$ having a plane model given by an affine equation of type 
$$f(y) = g(x),$$ where $f(T), g(T)\in \mathbb{F}_q[T]$ with $\deg(f(T))=a $ and $\deg(g(T))=b$ satisfying $\mbox{gcd}(a,b)=1$. Suppose that $\cX_{f,g}$ has genus $(a-1)(b-1)/2$ and is geometrically irreducible. Suppose moreover that there exist $a+1$ distinct rational points $P_1, P_2, \ldots, P_{a+1}$ on $\cX_{f,g}$ such that
\begin{equation}\label{eq1}
a P_1 \sim P_2 + \cdots + P_{a+1}
\end{equation}
and
\begin{equation}\label{eq2}
\ b P_1 \sim b P_{j} \ \mbox{ for } j\in \{2,\ldots,a+1\},
\end{equation}
where $b$ is the smallest positive integer satisfying (\ref{eq2}) and ``$\sim$" represents the linear equivalence of divisors. Notice that, by the above assumptions, $H(P_1)=\langle a,b\rangle$.
Many notable curves over finite fields---as Hermitian curves, Norm-Trace curves, curves from Kummer extensions, among others---arise in the family under consideration and satisfy the  required conditions.

In this section we will compute $\hGamma(\mathbf{P}_m)\cap \cC_m$, where $\mathbf{P}_m:=(P_1, \ldots, P_m)$ for $2\leq m\leq a+1$, which, according to Theorem \ref{maximals}, determines entirely $\hGamma(\mathbf{P}_m)$, the generating set of $\hH(\mathbf{P}_m)$. For this purpose, we assume furthermore that $ q \geq a+1$. 

\begin{lemma} \label{discrep} Let $2\leq m\leq a+1$ and $1\leq i<b$. The divisor 
$$A=(a(b-i)-b(m-1))P_1+iP_2+\cdots+iP_m$$ is a discrepancy with respect to any pair of distinct points in $\{P_1,\ldots,P_m\}$.
\end{lemma}
\begin{proof}
From \eqref{eq1} and \eqref{eq2} there exist functions $h,g_2,\ldots,g_{a+1}\in \Fq(\cX_{f,g})$ such that
\begin{equation} \label{div}
\divv(h)=\sum_{k=2}^{a+1} P_k-aP_1 \ \quad \mbox{and} \ \quad \divv(g_j)=bP_j-bP_1, \ \mbox{ for }j=2,\ldots, a+1.
\end{equation}
Since $b-i>0$, we conclude from \eqref{div} that $h^{b-i}/(g_2\cdots g_m)\in \cL(A)\backslash \cL(A-P_j)$, and we thus have $\cL(A)\neq \cL(A-P_j)$ for $j=1,\ldots,m$. To prove that $\cL(A-P) = \cL(A-P-Q)$ for any pair $P,Q\in\{P_1,\ldots,P_m\}$ with $P \neq Q$, we consider its equivalence $\cL(K+P+Q-A) \neq \cL(K+P-A)$, where $K$ is a canonical divisor on $\cX_{f,g}$. As $H(P_1)=\langle a,b\rangle$, we may take the canonical divisor $K=(ab-a-b-1)P_1$. So
$$
\begin{array}{rcl}
K+P+Q-A & = &(ab-a-b-1)P_1 + P + Q - (a(b-i)-b(m-1)) P_1-i\sum_{k=2}^m P_k\\
        & = &((i-1)a+(m-2)b-1)P_1 + P + Q - i\sum_{k=2}^m P_k.
\end{array}
$$

First, let us consider the case $P=P_1$ and $Q=P_j$ for $2\leq j\leq m$. Hence 
$$h^{i-1} \prod_{t=2\atop t\neq j}^m g_t \in \cL(K+P+Q-A) \backslash \cL(K+P-A)$$ and thus $\cL(A-P) = \cL(A-P-Q)$. Now, if $P=P_j$ and $Q=P_k$ for $2\leq j<k\leq m$, then

$$h^{i-1}\prod_{t=2\atop t\neq j,k}^m g_t \in \cL(K+P+Q-A) \backslash \cL(K+P-A)$$ and the result follows.
\end{proof}

Observe that from the aforementioned assumptions,  we have

$$\cC_m=\{\negalpha\in \ZZ^m \ : \ 0\leq \alpha_i<b \ \mbox{for } i=2,\ldots,m\}.$$
We also have that $\Theta_m$ is generated by the $m$-tuples
\begin{equation}\label{eta_i}
\negeta^i=(0,\ldots,0,\underbrace{-b}_{i\text{-th entry}},b,0,\ldots,0)\in \ZZ^m \quad \mbox{for} \quad i=2,\ldots,m;
\end{equation}
cf. page \pageref{teste}. The next theorem yields information about $\hGamma(\mathbf{P}_m)\cap \cC_m$, which will ensure a complete description of the generating set $\hGamma(\mathbf{P}_m)$ of $\hH(\mathbf{P}_m)$ and consequently of the Riemann-Roch spaces associated with divisors with support in $\{P_1,\ldots,P_m\}$.

\begin{theorem} \label{result} Let $P_1,\ldots,P_{a+1}$ be rational points on $\cX_{f,g}$ and $\cC_m$ be as above. For $2\leq m\leq  a+1$, let
$$\hS_m=\{(a(b-i)-b(m-1),i,\ldots,i)\in \ZZ^m \ : \ i=1,\ldots,b-1\}\cup \{{\bf 0}\}.$$
Then,  $$\hGamma(\mathbf{P}_m)\cap \cC_m=\hS_m.$$
\end{theorem}
\begin{proof} By Proposition \ref{prop equiv discrepancia} and Lemma \ref{discrep}, it follows that $\hS_m \subseteq \hGamma(\mathbf{P}_m) \cap \cC_m$. Hence it suffices to prove that $\hGamma(\mathbf{P}_m)\cap \cC_m\subseteq \hS_m$. For simplicity, let us denote
$$\negalpha^{i,m}=(a(b-i)-b(m-1),i,\ldots,i)\in \ZZ^m.$$
Thus $\hS_m=\{\negalpha^{i,m} \ : \ i=1,\ldots,b-1\} \cup \{\textbf{0}\}$. Notice that  $\alpha^{i,m}_1=\alpha^{i,m-1}_1-b$ for ${m\in \{3,\ldots,a+1\}}$. In order to prove that $\hGamma(\mathbf{P}_m)\cap \cC_m\subseteq \hS_m$, we will use induction on $m$. For $m=2$, it follows immediately from Lemma \ref{discrep}. Let now $m\geq 3$ and suppose that $\hGamma(\mathbf{P}_j)\cap \cC_j=\hS_j$ for $j=2,\ldots,m-1$. Given ${\negalpha=(\alpha_1,\ldots,\alpha_m)\in \hGamma(\mathbf{P}_m)\cap \cC_m}$, 
 let us consider ${s:=\min\{t\in \NN \ : \ \alpha_1+tb\geq 0\}}$. Observe that, by Proposition \ref{absmax}, $\ell(D_{\negalpha})\geq 1$, and thus $|\negalpha| = \deg(D_{\negalpha}) \geq 0$. In addition, $\negalpha \in \cC_m$ and hence $\alpha_1\geq -\sum_{j=2}^m \alpha_j\geq -(m-1)(b-1)$. So $\alpha_1+(m-1)b\geq 0$, which shows that $1\leq s \leq m-1$. 
 
If $1\leq s \leq m-2$, we have $2\leq m-s\leq m-1$. Without loss of generality, assume that $\alpha_2=\min_{2\leq j\leq m}\{\alpha_j\}$. Hence Theorem \ref{maximals} and equation \eqref{eta_i} lead to

\begin{equation*}
\negalpha'=(\alpha_1+sb,\alpha_2,\ldots,\alpha_{m-s},\alpha_{m-s+1}-b,\ldots,\alpha_m-b)\in \hH(\mathbf{P}_m)
\end{equation*}
Since $0\leq \alpha_j< b$ for $j=2,\ldots, m$, we get
$$\negbeta=\mbox{lub}(\negalpha',\textbf{0})=(\alpha_1+sb,\alpha_2,\ldots,\alpha_{m-s},0,\ldots,0)\in \hH(\mathbf{P}_{m-s}).$$
In particular, $\nabla_2^{m-s}(\negbeta)\neq \emptyset$. From the induction hypothesis, there exists $\negalpha^{i,m-s}\in \hS_{m-s}$ such that $\negalpha^{i,m-s}\in\nabla_2^{m-s}(\negbeta)$. As $\alpha_1^{i,m}=\alpha_1^{i,m-s}-sb$ and $\alpha_2=i\leq \alpha_j$ for $j=3,\ldots,m$, it follows that $\negalpha^{i,m}\in \nabla_2^{m}(\negalpha)$. Therefore, $\negalpha=\negalpha^{i,m}$ by Proposition \ref{absmax}, since $\negalpha\in \hGamma(\mathbf{P}_m)$. 

Now, if $s=m-1$, by using a similar argument, we obtain $\alpha_1+(m-1)b\in H(P_1)=\langle a,b \rangle$.  As $\alpha_1+(m-1)b<b$, we have either $\alpha_1+(m-1)b=0$ or $\alpha_1+(m-1)b=ak$ for $k\in \NN$. If $\alpha_1+(m-1)b=0$ then $\alpha_1=-(m-1)b$, which is a contradiction since $\alpha_1\geq -(m-1)(b-1)>-(m-1)b$. Therefore, $\alpha_1=a(b-i)-b(m-1)=\alpha^{i,m}_1$ for some $1\leq i< b$. Suppose that $\negalpha$ and $\negalpha^{i,m}$ are not comparable in the partial order $\leq$, since otherwise we would have a contradiction in the absolute maximality of $\negalpha$ and $\negalpha^{i,m}$. Hence, without loss of generality, we may assume that $\alpha_m<\alpha_m^{i,m}$. 
In this way, we have $\negalpha\in \nabla_1(\alpha^{i,m}_1,\alpha^{i,m}_2+b,\ldots,\alpha^{i,m}_{m-1}+b,\alpha^{i,m}_m)$. However, we claim that 
$$(*) \qquad \nabla_1(\alpha^{i,m}_1,\alpha^{i,m}_2+b,\ldots,\alpha^{i,m}_{m-1}+b,\alpha^{i,m}_m)=\emptyset.$$
Thus $\negalpha$ and $\negalpha^{i,m}$ are comparable in $\leq$, which leads to $\negalpha=\negalpha^{i,m}$ by their absolute maximal property. Therefore, $\hGamma(\mathbf{P}_m)\cap \cC_m\subseteq \hS_m$ and the proof is complete.
\end{proof}

\begin{proof}[Proof of $ \ (*)$:]
In view of  Theorem \ref{prop 2.1.3 da tese}(2), it is sufficient to prove that 
$$\ell\left(D_\negbeta-\sum_{j=2}^m P_j\right)=\ell\left(D_\negbeta-\sum_{j=1}^m P_j\right),$$ since $\nabla_1(\negbeta)=\nabla_1^m(\negbeta-\textbf{1}+\textbf{e}_1)$. Equivalently, let us show that 
$$\ell\left(K-D_\negbeta+\sum_{j=1}^m P_j\right)= \ell\left(K-D_\negbeta+\sum_{j=2}^m P_j\right)+1$$
for $K$ a canonical divisor on $\cX_{f,g}$. Let 
$${A'=(a(b-1)-b(m-1))P_1+\sum_{j=2}^{m-1}(i+b-1)P_{j}+(i-1)P_m}$$ and consider $K=(ab-a-b-1)P_1$. Hence
$$K-A'=((i-1)a+(m-2)b-1)P_1-(i+b-1)P_2-\cdots - (i+b-1)P_{m-1}-(i-1)P_m$$
and 
$$K-A'+P_1=((i-1)a+(m-2)b)P_1-(i+b-1)P_2-\cdots - (i+b-1)P_{m-1}-(i-1)P_m.$$
As
$$\divv\left(h^{i-1}\prod_{j=2}^{m-1} g_j\right)=-((i-1)a+(m-2)b))P_1+\sum_{j=2}^{m-1} (i+b-1)P_j+\sum_{j=m}^{a+1}(i-1)P_j,$$
we have $h^{i-1}\prod_{j=2}^{m-1} g_j\in \cL(K-A'+P_1)\backslash \cL(K-A')$, which proves the claim.
\end{proof}

In what follows, we include a concrete example to illustrate how simple are the elements  that determine the generalized Weierstrass semigroup $\hH(\negP_m)$.

\begin{example}\label{exemplo} Let $\ell$ be a prime power and let $r$ be an odd integer. Let $\cX$ be the curve defined over $\FF_{\ell^{2r}}$ by the affine equation
$$x^{\ell^r+1}=y^\ell+y.$$
This curve has genus $\ell^r(\ell-1)/2$. Let $P_\infty$ be the point at infinity $(0:1:0)$ of $\cX$ and let $P_{0b_j}$ be the points $(0:b_j:1)$ with $b_j\in \FF_{\ell^{2r}}$ in such a way that $b_j^\ell+b_j=0$ for $j=1,\ldots,\ell$. 
Observe that 
$$\divv(x)=\sum_{j=1}^\ell P_{0b_j}-\ell P_\infty \quad \mbox{and} \quad \divv(y-b_j)=(\ell^r+1)(P_{0b_j}-P_\infty) \ \mbox{for } j=1,\ldots,\ell.$$ Hence, $a=\ell$ and $b=\ell^r+1$. Therefore, according to Theorem \ref{result} and Theorem \ref{maximals}, the generalized Weierstrass semigroup $\hH(P_\infty,P_{0b_1},\ldots,P_{0b_{m-1}})$ for $2\leq m\leq \ell^r+1$, is completely determined by $\ell^r+m$ elements of $\hGamma(P_\infty,P_{0b_1},P_{0b_2},\ldots,P_{0b_{m-1}})\cap \cC_m$ and $\Theta(P_\infty,P_{0b_1},\ldots,P_{0b_{m-1}})$. 
For instance, if $\ell=5$, $r=1$, and $m=3$, the Hermitian curve of genus $10$,
$$(0, 0, 0), (13, 1, 1), (8, 2, 2),  (3, 3, 3),  (-2, 4, 4), (-7, 5, 5), (-6, 6, 0), \mbox{and }(0, -6, 6)$$
determine the generalized Weierstrass semigroup $\hH(P_\infty,P_{00},P_{0w^{15}})$, where $w$ is a primitive element of $\FF_{5^2}$. 
When $\ell=r=m=3$, the following 30 elements 

\begin{small}
$$(0, 0, 0), (25, 1, 1), (22, 2, 2), (19, 3, 3), (16, 4, 4), (13, 5, 5), (10, 6, 6), (7, 7, 7), (4, 8, 8), (1, 9, 9),$$ $$(-2, 10, 10), (-5, 11, 11), (-8, 12, 12), (-11, 13, 13), (-14, 14, 14), (-17, 15, 15), (-20, 16, 16),$$ $$(-23, 17, 17), (-26, 18, 18), (-29, 19, 19), (-32, 20, 20), (-35, 21, 21), (-38, 22, 22), (-41, 23, 23),$$ $$(-44, 24, 24), (-47, 25, 25), (-50, 26, 26), (-53, 27, 27), (-28, 28, 0), \mbox{and } (0,-28,28)$$
\end{small} determine entirely the generalized Weierstrass semigroup $\hH(P_\infty,P_{00},P_{0w^{546}})$, where $w$ is a primitive element of $\FF_{3^6}$. 
\end{example}

\section{Riemann-Roch spaces of divisors $D_\negalpha$ on $\cX_{f,g}$}

Let $\cX_{f,g}$ and $P_1,\ldots,P_{a+1}$ as in the previous section and under the same assumptions. Having established the elements that determine the generating set $\hGamma(\negP_{m})$ of the generalized Weierstrass semigroup $\hH(\negP_m)$, where ${\negP_m=(P_1,\ldots,P_m)}$  for $2\leq m\leq a+1$, we now use them to study the Riemann-Roch spaces $\mathcal{L}(D_{\negalpha})$ associated with the divisors $D_{\negalpha} = \alpha_1 P_1 + \cdots + \alpha_m P_m$ for $\negalpha=(\alpha_1,\ldots,\alpha_m)\in \ZZ^m$. From the elements of $\hGamma(\negP_{m})$, we deduce formulas for the dimension and also present a basis of $\mathcal{L}(D_{\negalpha})$. In particular, it allows us to study the supported floor of such divisors.

\begin{theorem} \label{dimension2}
Let $\negalpha = (\alpha_1, \ldots , \alpha_m) \in \ZZ^m$. Then,
$$\ell(D_{\negalpha} )=\sum_{i=1}^{b-1} \max\left\{\left\lfloor\frac{\alpha_1+ia-b(a+1-m)}{b}\right\rfloor+\sum_{j=2}^m \left\lfloor\frac{\alpha_j-i}{b} \right\rfloor+1,0\right\}$$
$$+ \max\left\{\sum_{j=1}^m \left\lfloor\frac{\alpha_j}{b} \right\rfloor+1,0\right\}.$$
\end{theorem}
\begin{proof} By Theorem \ref{maximals} and Theorem \ref{result}, we have that the elements of $\hGamma(\negP_m)$ are exactly the $m$-tuples of the following types:

(I) $ \ (-ai+b(a+1-m-t_1),i+b(t_1-t_2),\ldots,i+b(t_{m-2}-t_{m-1}),i+bt_{m-1})$;

(II) $ \ (-bt_1,b(t_1-t_2),\ldots,b(t_{m-2}-t_{m-1}),bt_{m-1})$,

\noindent where $1\leq i<b$ and $t_j\in \ZZ$ for $j=1,\ldots,m-1$. As Theorem \ref{dimension} gives ${\ell(D_{\negalpha} )=\#\left(\hGamma_{\negP_m}(\negalpha)/\equiv_1\right)}$, we will count ${\#\{\beta_1 \ : \ \negbeta=(\beta_1,\ldots,\beta_m)\in \hGamma_{\negP_m}(\negalpha)\}}$.

To do that, let us first count $\#\{\beta_1 \ : \ \negbeta=(\beta_1,\ldots,\beta_m)\in \hGamma_{\negP_m}(\negalpha) \ \mbox{of type (I)}\}$, that is,
\begin{equation} \label{eq:absmaxalphaI}
\begin{array}{l}
\beta_1=-ai+b(a+1-m-t_1)  \leq   \alpha_1,\\
\beta_j=i+b(t_{j-1}-t_j) \leq  \alpha_j , \qquad \mbox{for} \qquad j=2,\ldots,m-1, \\
\beta_m=i+bt_m\leq \alpha_m.
\end{array}
\end{equation}
Writing $d_j=t_{j-1}-t_j$ for $j=2,\ldots,m-1$, and $d_m=t_{m-1}$, we obtain $t_1=\sum_{j=2}^{m}d_j$. We may thus rewrite \eqref{eq:absmaxalphaI} as
\begin{equation} \label{eq:absmaxe}
\begin{array}{l}
\beta_1=-ai+b(a+1-m-\sum_{j=2}^m d_j)  \leq   \alpha_1\\
\beta_j=i+bd_j \leq  \alpha_j  \qquad \mbox{for} \quad j=2,\ldots,m.
\end{array}
\end{equation}
Hence we must compute the number of distinct values ${-ai+b(a+1-m-\sum_{j=2}^m d_j)}$, where the $d_j$'s satisfy the constraints in \eqref{eq:absmaxe}. This number equals to number of integers $\sum_{j=2}^m d_j$ since $\gcd(a,b)=1$. From \eqref{eq:absmaxe} we obtain $i+bd_j \leq  \alpha_j$ for $j=2,\ldots,m$, and thus $d_j\leq \lfloor \frac{\alpha_j-i}{b}\rfloor$. It implies that $\sum_{j=2}^m d_j\leq \sum_{j=2}^m \lfloor \frac{\alpha_j-i}{b}\rfloor$. From \eqref{eq:absmaxe} we also get
$$\sum_{j=2}^m d_j \geq \left\lceil -\frac{\alpha_1+ai-b(a+1-m)}{b}\right\rceil=-\left\lfloor \frac{\alpha_1+ai-b(a+1-m)}{b}\right\rfloor.$$
Therefore, the integers $\sum_{j=2}^m d_j$ satisfy
\begin{equation*}\label{eq:aux}
-\left\lfloor \frac{\alpha_1+ai-b(a+1-m)}{b}\right\rfloor\leq \sum_{j=2}^m d_j \leq \sum_{j=2}^m \left\lfloor \frac{\alpha_j-i}{b}\right\rfloor,
\end{equation*}
and the number of such integers is
\begin{equation} \label{eq:sum1}
\max\left\{\left\lfloor\frac{\alpha_1+ia-b(a+1-m)}{b}\right\rfloor+\sum_{j=2}^m \left\lfloor\frac{\alpha_j-i}{b} \right\rfloor+1,0\right\}.
\end{equation}

Finally, it remains to count $\#\{\beta_1 \ : \ \negbeta=(\beta_1,\ldots,\beta_m)\in \hGamma_{\negP_m}(\negalpha) \ \mbox{of type (II)}\}$, that is,
\begin{equation} \label{eq:absmaxalphaII}
\begin{array}{l}
\beta_1=-bt_1  \leq   \alpha_1,\\
\beta_j=b(t_{j-1}-t_j) \leq  \alpha_j , \qquad \mbox{for} \qquad j=2,\ldots,m-1, \\
\beta_m=bt_m\leq \alpha_m.
\end{array}
\end{equation}
Again, writing $d_j=t_{j-1}-t_j$ for $j=2,\ldots,m-1$, and $d_m=t_{m-1}$, we may rephrase \eqref{eq:absmaxalphaII} as
\begin{equation} \label{eq:absmaxe2}
\begin{array}{l}
\beta_1=-b\sum_{j=2}^m d_j  \leq   \alpha_1\\
\beta_j=bd_j \leq  \alpha_j  \qquad \mbox{for} \quad j=2,\ldots,m.
\end{array}
\end{equation}
Proceeding in the same way, we conclude that $\#\{\beta_1 \ : \ \negbeta=(\beta_1,\ldots,\beta_m)\in \hGamma_{\negP_m}(\negalpha) \ \mbox{of type (II)}\}$ is equal to the number of integers $\sum_{j=2}^m d_j$ satisfying
\begin{equation*}\label{eq:aux}
-\left\lfloor \frac{\alpha_1}{b}\right\rfloor\leq \sum_{j=2}^m d_j \leq \sum_{j=2}^m \left\lfloor \frac{\alpha_j}{b}\right\rfloor,
\end{equation*}
which is
\begin{equation} \label{eq:sum2}
\max\left\{\sum_{j=1}^m \left\lfloor\frac{\alpha_j}{b} \right\rfloor+1,0\right\}.
\end{equation}
We thus obtain the claimed formula by summing the values \eqref{eq:sum2} and \eqref{eq:sum1} for $1\leq i<b$.
\end{proof}

A simpler and general expression for the dimension of a divisor with support contained in $\{P_1,\ldots,P_{a+1}\}$ can be deduced by using the case ${m=a+1}$.

\begin{corollary} Let $A=\alpha_1P_1+\cdots +\alpha_{a+1}P_{a+1} \in \mbox{Div}(\cX_{f,g})$. Then,
$$\ell(A)=\sum_{i=0}^{b-1} \max\left\{\left\lfloor\frac{\alpha_1+ia}{b}\right\rfloor+\sum_{j=2}^{a+1} \left\lfloor\frac{\alpha_j-i}{b} \right\rfloor+1,0\right\}.$$
\end{corollary}

We now employ the elements of the generating set $\hGamma(\negP_m)$ of $\hH(\negP_m)$ to provide a basis for the Riemann-Roch spaces associated with divisors $D_\negalpha$ for $\negalpha\in \ZZ^m$.

\begin{corollary}\label{corbasis} Let $\negalpha=(\alpha_1,\ldots,\alpha_m)\in \ZZ^m$ and let $h, g_2, \ldots, g_m$ be the functions given in \eqref{div}. Then,
$$\bigcup_{i=1}^{b-1} \ \left\{\frac{h^{b-i}}{g_2^{d_2+1}\cdots g_m^{d_m+1}} \ : {d_j\in \ZZ  \ \mbox{with } \ d_j\leq \left\lfloor \frac{\alpha_j-i}{b}\right\rfloor \ \mbox{for } j=2,\ldots,m, \ \mbox{and} \atop -\left\lfloor \frac{\alpha_1+ai-b(a+1-m)}{b}\right\rfloor\leq \sum_{j=2}^m d_j \leq \sum_{j=2}^m \left\lfloor \frac{\alpha_j-i}{b}\right\rfloor}\right\}$$
$$\bigcup \ \left\{\frac{1}{g_2^{d_2}\cdots g_m^{d_m}} \ : {d_j\in \ZZ \ \ \mbox{with } \ d_j\leq \left\lfloor \frac{\alpha_j}{b}\right\rfloor \ \mbox{for } j=2,\ldots,m, \ \mbox{and} \atop -\left\lfloor \frac{\alpha_1}{b}\right\rfloor\leq \sum_{j=2}^m d_j \leq \sum_{j=2}^m \left\lfloor \frac{\alpha_j}{b}\right\rfloor}\right\}$$
is a base of $\cL(D_{\negalpha})$.
\end{corollary}
\begin{proof} Notice from \eqref{eq:absmaxe} and \eqref{eq:absmaxe2} in the proof of Theorem \ref{dimension2} that the representative classes for the elements in $\hGamma_{\mathbf{P}_m}(\negalpha)/\equiv_1$ are given by $m$-tuples in either of the types
$$\left(-ai+b(a+1-m-\textstyle\sum_{j=2}^m d_j),i+bd_2,\ldots,i+bd_m\right),$$
where $d_j\in \ZZ$ for $j=2,\ldots,m$ and satisfy
$$\textstyle-\left\lfloor \frac{\alpha_1+ai-b(a+1-m)}{b}\right\rfloor\leq \sum_{j=2}^m d_j \leq \sum_{j=2}^m \left\lfloor \frac{\alpha_j-i}{b}\right\rfloor \ \ \mbox{with } \ d_j\leq \left\lfloor \frac{\alpha_j-i}{b}\right\rfloor \ \mbox{for } j=2,\ldots,m,$$
and
$$\left(-b\textstyle\sum_{j=2}^m d_j,bd_2,\ldots,bd_m\right),$$
where $d_j\in \ZZ$ for $j=2,\ldots,m$ and satisfy
$$-\left\lfloor \frac{\alpha_1}{b}\right\rfloor\leq \sum_{j=2}^m d_j \leq \sum_{j=2}^m \left\lfloor \frac{\alpha_j}{b}\right\rfloor \ \ \mbox{with } \ d_j\leq \left\lfloor \frac{\alpha_j}{b}\right\rfloor \ \mbox{for } j=2,\ldots,m.$$
For $i=1,\ldots,b-1$, as
\begin{small}
$$\divv\left(\frac{h^{b-i}}{g_2^{d_2+1}\cdots g_m^{d_m+1}}\right)=-\sum_{j=2}^m (i+bd_j)P_j+(b-i)\sum_{j=m+1}^{a+1} P_j - \left(a(b-i)-b(m-1)-b\textstyle\sum_{j=2}^m d_j\right)P_1,$$
\end{small}
we thus have
\begin{equation*}
\rho\left(\frac{h^{b-i}}{g_2^{d_2+1}\cdots g_m^{d_m+1}}\right)=\left(-ai+b(a+1-m-\textstyle\sum_{j=2}^m d_j),i+bd_2,\ldots,i+bd_m\right),
\end{equation*}
where $\rho(\cdot)=(v_{P_1}(\cdot),\ldots,v_{P_m}(\cdot))$. Similarly, since
$$\divv\left(\frac{1}{g_2^{d_2}\cdots g_m^{d_m}}\right)= (\textstyle\sum_{j=2}^m bd_j) P_1 -\displaystyle\sum_{j=2}^m bd_jP_j,$$
we obtain
$$\rho\left(\frac{1}{g_2^{d_2}\cdots g_m^{d_m}}\right)=\left(-b\textstyle\sum_{j=2}^m d_j,bd_2,\ldots,bd_m\right).$$
The proof is therefore complete by invoking Corollary \ref{basis}.
\end{proof}

\begin{example}\label{exemplo2}
Consider the curve $x^{28}=y^3+y$ over $\FF_{3^6}$ in the family of curves in Example \ref{exemplo} (the case $\ell=r=3$). In this case, $a=3$ and $b=28$. Taking $m=3$, let us consider $\negalpha=(8,7,-1)$ and the corresponding divisor $D_\negalpha=8P_\infty+7P_{00}-1P_{0w^{546}}$, where $w$ is a primitive element of $\FF_{3^6}$. As there is no solution $d_2,d_3\in \ZZ$ to
$$\textstyle 0=-\left\lfloor \frac{8}{28}\right\rfloor\leq d_2+d_3 \leq \left\lfloor \frac{7}{28}\right\rfloor+\left\lfloor \frac{-1}{28}\right\rfloor=-1,$$
and the only $i\in \{1,\ldots, 27\}$ such that possibly there are $d_2,d_3\in \ZZ$ with
$$\textstyle-\left\lfloor \frac{3i-20}{28}\right\rfloor\leq d_2+d_3 \leq \left\lfloor \frac{7-i}{28}\right\rfloor+\left\lfloor \frac{-1-i}{28}\right\rfloor$$
are $26$ and $27$, we must find $d_2,d_3\in \ZZ$ in each of the following cases:

\begin{itemize}
\item[$\diamond$] $d_2+d_3=-2$ with $d_2\leq \lfloor\frac{7-26}{28}\rfloor=-1$ and  $d_3\leq  \lfloor\frac{-1-26}{28}\rfloor=-1$; and
\item[$\diamond$] $d_2+d_3=-2$ with $d_2\leq \lfloor\frac{7-27}{28}\rfloor=-1$ and  $d_3\leq  \lfloor\frac{-1-27}{28}\rfloor=-1$.
\end{itemize}
In both cases, the only solution is $d_2=-1$ and $d_3=-1$. Applying these values in the previous result, we obtain that $x^2$ and $x$ forms a basis for the Riemann-Roch space $\cL(D_\negalpha)$.
\end{example}

We can also provide a formula to the $\negP_m$-floor of divisors $D_\negalpha$ for $\negalpha\in \ZZ^m$. To do this, let us consider $N_0:=\max\{\sum_{j=1}^m \lfloor \frac{\alpha_j}{b}\rfloor+1,0 \}$ and 
$$N_i:=\max\left\{\textstyle\left\lfloor\frac{\alpha_1+ia-b(a+1-m)}{b}\right\rfloor+\sum_{j=2}^m \left\lfloor\frac{\alpha_j-i}{b} \right\rfloor+1,0\right\} \ \mbox{for } i=1,\ldots,b-1.$$ 

\begin{proposition}\label{supfloorX}
Let $\negalpha=(\alpha_1,\ldots,\alpha_m)\in \ZZ^m$. Then, 
$$\lfloor D_{\negalpha} \rfloor_{\mathbf{P}_m}  = \alpha_1 'P_1 + \cdots + \alpha_m 'P_m,$$ where
\begin{small}
$$\alpha'_1=\max\left(\left\{ -ai+b\left(a+1-m+\textstyle\left\lfloor \frac{\alpha_1+ai-b(a+1-m)}{b}\right\rfloor \right) \ : \ {i=1,\ldots,b-1 \atop \mbox{with } \ N_i\neq 0}\right\}\bigcup \left\{b\left \lfloor \frac{\alpha_1}{b}\right\rfloor \ : \ N_0\neq 0\right\}\right)$$
and
$$\alpha'_j=\max\left\{i+b\textstyle\left\lfloor \frac{\alpha_j-i}{b}\right\rfloor \ : \ i=0,\ldots,b-1 \ \mbox{with } N_i\neq 0\right\}  \ \mbox{for} \  j=2,\ldots,m.$$
\end{small}
\end{proposition}
\begin{proof}
It follows from the proof of Theorem \ref{dimension2} that each $\negbeta=(\beta_1,\ldots,\beta_m)\in \hGamma_{\mathbf{P}_m}(\negalpha)$ is expressed in one of the forms \eqref{eq:absmaxe} or \eqref{eq:absmaxe2}. Since Proposition \ref{floor} yields $\lfloor D_{\negalpha} \rfloor_{\mathbf{P}_m}=D_{\text{lub}(\hGamma_{\negP_m}(\negalpha))}$ and each coordinate of $\text{lub}(\hGamma_{\negP_m}(\negalpha))$ is the biggest one among the corresponding coordinates of elements $\negbeta\in \hGamma_{\mathbf{P}_m}(\negalpha)$, we have the desired formula.
\end{proof}

In particular, we get a formula for the floor of divisors with support in $\{P_1,\ldots,P_{a+1}\}$. Similarly, we let $\tilde{N}_i:=\max\left\{\left\lfloor\frac{\alpha_1+ia}{b}\right\rfloor+\sum_{j=2}^m \left\lfloor\frac{\alpha_j-i}{b} \right\rfloor+1,0\right\}$ for $i=0,\ldots,b-1$.

\begin{corollary}\label{floorX}
Let $A=\alpha_1P_1+\cdots +\alpha_{a+1}P_{a+1}\in \mbox{Div}(\cX_{f,g})$. Then, 
$$\lfloor A \rfloor=\alpha'_1P_1+\cdots +\alpha'_{a+1}P_{a+1},$$ where
$$\alpha'_1=\max\left\{-ai+b\textstyle\left\lfloor \frac{\alpha_1+ai}{b}\right\rfloor \ : \ i=0,1,\ldots,b-1 \ \mbox{with } \tilde{N}_i\neq 0\right\} \ \mbox{and}$$
$$\alpha'_j=\max\left\{ i+b\textstyle\left\lfloor \frac{\alpha_j-i}{b}\right\rfloor \ : \ i=0,1,\ldots,b-1 \ \mbox{with } \tilde{N}_i\neq 0\right\}  \ \mbox{for } \  j=2,\ldots,a+1.$$
\end{corollary}
\begin{proof}
Notice that, by the previous result, the formula is exactly the $\negP_m$-floor of $A$ for $m=a+1$. From Corollary \ref{corbasis}, each function in the spanning set of $\cL(A)$ has no zeroes outside the set $\{P_1,\ldots,P_{a+1}\}$. Therefore the result follows by Proposition \ref{mflooreqfloor}.
\end{proof}

\begin{example}
Consider again the curve $x^{28}=y^3+y$ over $\FF_{3^6}$ of Example~\ref{exemplo2} and let $w$ be a primitive element of $\FF_{3^6}$. Here we will see the difference between the supported floor and the floor of a divisor. For $m=3$, let $\negP_3=(P_\infty, P_{00}, P_{0w^{546}})$. As in Example~\ref{exemplo2}, the only $i\in \{0,\ldots,27\}$ such that $N_i\neq 0$ are $26$ and $27$, we have, according to Proposition \ref{supfloorX}, that the $\negP_3$-floor of the divisor $D_\negalpha=8P_\infty+7P_{00}-P_{0w^{546}}$, associated with
$\negalpha=(8,7,-1)$, is
$$\lfloor D_\negalpha\rfloor_{\negP_3}=6P_\infty-P_{00}-P_{0w^{546}},$$
where $6=\max\{6,3\}$, $-1=\max\{-2,-1\}$, and $-1=\max\{-2,-1\}$ in the formula. In a similar way, we have that the only $i\in \{0,\ldots,27\}$ such that $\tilde{N}_i\neq 0$ are $26$ and $27$, which from Corollary~\ref{floorX} gives us that the floor of $D_\negalpha$ is
$$\lfloor D_\negalpha\rfloor=6P_\infty-P_{00}-P_{0w^{546}}-P_{0w^{182}},$$
with $6=\max\{6,3\}$, $-1=\max\{-2,-1\}$, $-1=\max\{-2,-1\}$, and $-1=\max\{-2,-1\}$.
\end{example}

\section*{Acknowledgments}

The first author was supported by CAPES-Brazil. The second author was supported by CNPq-Brazil under grant 446913/2014-6 and by FAPEMIG-Brazil under grant APQ-01607-14.

\end{document}